\RequirePackage{fix-cm}
\documentclass[smallextended]{svjour3}       % onecolumn (second format)
\smartqed  % flush right qed marks, e.g. at end of proof
\usepackage{graphicx}
\usepackage{amssymb}    
\usepackage{amsmath}   
\usepackage[all,line,
]{xy} 
\CompileMatrices

\newcommand{\Span}{\operatorname{Span}} 
\newcommand{\Wan}{\operatorname{Wan}}
\newcommand{\Aut}{\operatorname{Aut}}
\newcommand{\Tr}{\operatorname{Tr}}
\newcommand{\diag}{\operatorname{diag}}
\newtheorem{obs}[theorem]{Observation}
\begin{document}

\title{The factor type of dissipative KMS weights on graph $C^*$-algebras
}


\author{Klaus Thomsen
}


\institute{K. Thomsen \at
              Department of Mathematics, Aarhus University, Ny Munkegade, 8000 Aarhus C, Denmark \\
              \email{matkt@math.au.dk}           %  \\
           %\and
           %S. Author \at
           %   second address
}

\date{Received: date / Accepted: date}

\maketitle

\begin{abstract}
  We calculate the S-invariant of Connes for the von Neumann algebra factors arising from KMS weights of a generalized gauge action on a simple graph C*-algebra when the associated measure on the infinite path space of the graph is dissipative under the action of the shift.

\end{abstract}

%\section{Introduction}
%\label{intro}

\section{Introduction} 
The purpose with this paper is to determine the factor types of a class of KMS weights and KMS states that arise from generalized gauge actions on simple graph $C^*$-algebras. In previous work the author has studied such KMS weights and among other things it was shown that they can be naturally divided into three classes depending on the properties of the measure they induce on the unit space of the groupoid underlying the $C^*$-algebra. For the KMS weights we consider here the measure is concentrated on the space of infinite paths in the graph, and for those there is a dichotomy; the measure is either concentrated on the paths that contain all arrows infinitely many times or on the set of paths that eventually leave every finite set of vertexes. The first case can at most arise for one inverse temperature in which case the measure is conservative for the shift map and the KMS weight is unique up to multiplication by scalars. In \cite{Th5} the corresponding KMS weights were called conservative and their factor types were determined. In the second case the measure is dissipative for the shift and we call them therefore dissipative KMS weights. There are almost always many more dissipative KMS weights than conservative ones when the graph is infinite. It is the purpose of the present paper to determine the factor types of the extremal dissipative KMS weights. This will be done by calculating the $S$-invariant of Connes for the associated factors. It was shown in \cite{Th2} that there is a bijective correspondence between the rays of KMS weights and the KMS states on the corner defined by a vertex in the graph, and in \cite{Th5} it was shown that the factor generated by the GNS construction of the state is a corner in the factor of the corresponding KMS weight. Hence the two von Neumann algebra factors have the same $S$-invariant and our calculation covers both cases. 

In the preceding work the author developed methods for the calculation of the $\Gamma$-invariant of Connes for the factors arising from KMS weights on graph $C^*$-algebras, see Section 4.4 of \cite{Th1}, Theorem 4.1 in \cite{Th3} and Theorem 3.3 in \cite{Th5}. These methods are difficult to apply to general dissipative KMS weights. In addition, unlike the conservative KMS weights, dissipative KMS weights give rise to both purely infinite and semi-finite factors and it is therefore more appropriate to use the $S$-invariant since it can distinguish semi-finite factors from type $III_0$ factors which the $\Gamma$-invariant can not. To this end, inspired by the work of Laca, Larsen, Neshveyev, Sims and Webster in \cite{LLNSW}, we will here use the methods and results from the papers by Feldman and Moore, \cite{FM1},\cite{FM2}. Compared to \cite{LLNSW} a major difference is that in the present setting the fixed point algebra of the modular automorphism group is not a factor in general.  

The main new ingredient we introduce in order to formulate and prove the results is a closed sub-semigroup of the multiplicative semigroup of non-negative real numbers which we associate to every wandering path in the graph by using the potential defining the generalized gauge action. We call this the detour semigroup of the path. It follows from results in \cite{Th4} that the conformal measure defined by an extremal dissipative $\beta$-KMS weight is concentrated on wandering paths with the same detour semigroup, and the main result is obtained by proving that the $S$-invariant of the corresponding factor is the detour semigroup raised to the power $\beta$.

The last section of the paper is devoted to applications and examples. In particular, the main result is used to determine which factors arise from the KMS states of the generalized gauge actions on Cayley graphs studied in \cite{CT2}. It follows that these factors are all of type $III_{\lambda}$ for some $\lambda \neq 0$, exactly as the factors arising from conservative KMS weights, cf. \cite{Th5}. To show that dissipative KMS weights can also give rise to factors of type $I$, $II$ and $III_0$ we study the dissipative KMS weights for generalized gauge actions on graphs with $\beta$-summable exits as they were introduced in \cite{Th2}. The definition of the detour semigroup resembles the definition of the ratio set of Kriger, \cite{Kr}, which in turn was derived from the asymptotic ratio set of Araki and Woods, \cite{AW}, and we give more substance to this similarity by showing that for strongly connected graphs with countably many exits the factor of an extremal dissipative $\beta$-KMS weight contains a corner which is an Araki-Woods factor whose asymptotic ratio set in many cases is the same as the relevant detour semigroup raised to the power $\beta$. This connection to Araki-Woods factors allows us to easily exhibit factors of type $I$, type $II$ and type $III_0$ arising from dissipative KMS weights.

%\label{sec:1}

\section{Factors from KMS weights}

Let $A$ be a $C^*$-algebra and $A_{+}$ the convex cone of positive
elements in $A$. A \emph{weight} on $A$ is a map $\psi : A_{+} \to
[0,\infty]$ with the properties that $\psi(a+b) = \psi(a) + \psi(b)$
and $\psi(\lambda a) = \lambda \psi(a)$ for all $a, b \in A_{+}$ and all
$\lambda \in \mathbb R, \ \lambda > 0$. By definition $\psi$ is
\emph{densely defined} when $\left\{ a\in A_{+} : \ \psi(a) <
  \infty\right\}$ is dense in $A_{+}$ and \emph{lower semi-continuous}
when $\left\{ a \in A_{+} : \ \psi(a) \leq \alpha \right\}$ is closed
for all $\alpha \geq 0$. We refer to \cite{Ku}, \cite{KV1} and \cite{KV2} for more information on weights, and as in \cite{KV2} we say that a weight is \emph{proper}
when it is non-zero, densely defined and lower semi-continuous. Let $\psi$ be a proper weight on $A$. Set $\mathcal N_{\psi} = \left\{ a \in A: \ \psi(a^*a) < \infty
\right\}$ and note that 
\begin{equation*}\label{f3}
\mathcal N_{\psi}^*\mathcal N_{\psi} = \Span \left\{ a^*b : \ a,b \in
  \mathcal N_{\psi} \right\}
\end{equation*} 
is a dense
$*$-subalgebra of $A$, and that there is a unique well-defined linear
map $\mathcal N_{\psi}^*\mathcal N_{\psi} \to \mathbb C$ which
extends $\psi : \mathcal N_{\psi}^*\mathcal N_{\psi} \cap A_+ \to
[0,\infty)$. We denote also this densely defined linear map by $\psi$.

Let $\alpha : \mathbb R \to \Aut A$ be a point-wise
norm-continuous one-parameter group of automorphisms on
$A$. Let $\beta \in \mathbb R$. Following \cite{C} we say that a proper weight
$\psi$ on $A$ is a \emph{$\beta$-KMS
  weight} for $\alpha$ when
\begin{enumerate}
\item[i)] $\psi \circ \alpha_t = \psi$ for all $t \in \mathbb R$, and
\item[ii)] for every pair $a,b \in \mathcal N_{\psi} \cap \mathcal
  N_{\psi}^*$ there is a continuous and bounded function $F$ defined on
  the closed strip $D_{\beta}$ in $\mathbb C$ consisting of the numbers $z \in \mathbb C$
  whose imaginary part lies between $0$ and $\beta$, and is
  holomorphic in the interior of the strip and satisfies that
$$
F(t) = \psi(a\alpha_t(b)), \ F(t+i\beta) = \psi(\alpha_t(b)a)
$$
for all $t \in \mathbb R$.% \footnote{Note that we apply the
  %definition from \cite{C} for the action $\alpha_{-t}$
 % in order to use the same sign conention as in \cite{BR}, for example.}
\end{enumerate}   
A $\beta$-KMS weight $\psi$ with the property that 
$$
\sup \left\{ \psi(a) : \ 0 \leq a \leq 1 \right\} = 1
$$
will be called a \emph{$\beta$-KMS state}. A $\beta$-KMS weight $\psi$ on $A$ is \emph{extremal} when the only
$\beta$-KMS weights $\varphi$ on $A$ with the property that
$\varphi(a) \leq \psi(a)$ for all $a \in A_+$ are scalar
multiples of $\psi$, viz. $\varphi = s\psi$ for some $s > 0$.  

Given a proper weight $\psi$ on a $C^*$-algebra $A$ there is a GNS-type
construction consisting of a Hilbert space $H_{\psi}$, a linear map
$\Lambda_{\psi} : \mathcal N_{\psi} \to H_{\psi}$ with dense range and
a non-degenerate representation $\pi_{\psi}$ of $A$ on $H_{\psi}$ such that
\begin{enumerate}
\item[$\bullet$] $\psi(b^*a) = \left<
    \Lambda_{\psi}(a),\Lambda_{\psi}(b)\right>, \ a,b \in \mathcal
  N_{\psi}$, and
\item[$\bullet$] $\pi_{\psi}(a)\Lambda_{\psi}(b) = \Lambda_{\psi}(ab), \
  a \in A, \ b \in \mathcal N_{\psi}$,
\end{enumerate} 
cf. \cite{Ku}, \cite{KV1}, \cite{KV2}. When $\psi$ is an extremal $\beta$-KMS weight the von Neumann algebra 
$$
M(\psi) = \pi_{\psi}(A)''
$$ 
generated by $\pi_{\psi}(A)$ is a factor, cf. Lemma 4.9 in \cite{Th1}. We aim to determine the factor type of $M(\psi)$ for the cases we consider by calculating the $S$-invariant $S(M(\psi))$ of Connes; see Section III in \cite{Co}. To this end we use that $\psi$ extends to a normal semi-finite faithful weight
$\tilde{\psi}$ on $M(\psi)$ such that $\psi =
\tilde{\psi} \circ \pi_{\psi}$, and that the modular group on
$M(\psi)$ corresponding to $\tilde{\psi}$ is the
one-parameter group $\theta$ given by
$$
\theta_t = \tilde{\alpha}_{-\beta t} \ ,
$$
where $\tilde{\alpha}$ is the
$\sigma$-weakly continuous extension of $\alpha$ defined such that
$\tilde{\alpha}_t \circ \pi_{\psi} = \pi_{\psi} \circ \alpha_t$, cf. Section 2.2 in \cite{KV1}.

Assume now that $p\in A$ is projection in the fixed point algebra of $\alpha$, and that $p$ is full in $A$. Let $(\pi_{\phi},H_{\phi}, u_{\phi})$ be the GNS-representation of the state $\phi$ on $pAp$ defined such that
$$
\phi(x) = \psi(p)^{-1}\psi(x) \ ,
$$ 
cf. Theorem 2.4 in \cite{Th2}.
By Lemma 2.3 in \cite{Th5} there is an $*$-isomorphism 
\begin{equation*}\label{07-03-18a}
\pi_{\phi}\left(pAp\right)'' \ \simeq \ \pi_{\psi}(p)M(\psi)\pi_{\psi}(p)
\end{equation*}
of von Neumann algebras. 
It follows therefore from \cite{Co} that 
$$
 S\left(M(\psi)\right) =S\left(\pi_{\phi}\left(pAp\right)'' \right)  \ .
$$
This realisation of $S\left(M(\psi)\right)$ has many technical advantages because $\phi$ is a state and it will be used below.

\subsection{Generalized gauge actions on graph $C^*$-algebras}
Let $\Gamma$ be a countable directed graph with vertex set $\Gamma_V$ and arrow set
$\Gamma_{Ar}$. For an arrow $a \in \Gamma_{Ar}$ we denote by $s(a) \in \Gamma_V$ its source and by
$r(a) \in \Gamma_V$ its range. A vertex $v \in \Gamma_V$ which does not emit any arrow is
a \emph{sink}, while a vertex $v$ which emits infinitely many arrows is called an \emph{infinite
  emitter}. An \emph{infinite path} in $\Gamma$ is an element
$p = (p_i)_{i=1}^{\infty} \in \left(\Gamma_{Ar}\right)^{\mathbb N}$ such that $r(p_i) = s(p_{i+1})$ for all
$i$.  A finite path $\mu = a_1a_2 \cdots a_n = (a_i)_{i=1}^n \in \left(\Gamma_{Ar}\right)^n$ is
defined similarly. The number of arrows in $\mu$ is its \emph{length}
and we denote it by $|\mu|$. A vertex $v \in \Gamma_V$ will be considered as a finite path of length $0$. We let $P(\Gamma)$ denote the set of
infinite paths in $\Gamma$ and $P_f(\Gamma)$ the set of finite paths in $\Gamma$. We extend the source map to $P(\Gamma)$ such that
$s(p) = s(p_1)$ when $p = \left(p_i\right)_{i=1}^{\infty} \in P(\Gamma)$, and the
range and source maps to $P_f(\Gamma)$ such that $s(\mu) = s(a_1)$ and $r(\mu)
= r(a_{|\mu|})$ when $|\mu|\geq 1$, and set $s(v) = r(v) = v$ when $v\in \Gamma_V$. When $\mu$ is a finite path and $x$ is either a finite or infinite path with $r(\mu) = s(x)$ we can define the concatenation $\mu x$ in the obvious way, the result being a finite or infinite path depending on $x$.

The $C^*$-algebra $C^*(\Gamma)$ of the graph $\Gamma$ was introduced in this generality in \cite{FLR} as the universal
$C^*$-algebra generated by a collection of partial isometries and projections subject to conditions determined by $\Gamma$, but in the present work we shall depend on its realisation as the $C^*$-algebra of a groupoid. When $\Gamma$ is row-finite there is a description of the relevant groupoid $\mathcal G$ in \cite{KPRR} and for a general directed graph it was described by Paterson in \cite{Pa}. The unit space $\mathcal G^0$ of the groupoid $\mathcal G$ is the union
$$
\mathcal G^0 = P(\Gamma) \cup Q(\Gamma), 
$$
where $Q(\Gamma)$ is the set of finite paths that terminate at a vertex which is either a sink or an infinite emitter. Note that $Q(\Gamma)$ is countable.  Associated to the finite path $\mu \in P_f(\Gamma)$ is the cylinder set 
$$
Z(\mu) = \left\{(p_i)_{i=1}^{\infty} \in \mathcal G^0 : \  
p_j = a_j, \ j = 1,2, \cdots, |\mu|\right\}  \ .
$$
In particular, when $\mu$ has length $0$ and hence is just a vertex $v$,
$$
Z(v) = \left\{ p \in \mathcal G^0 :  \  s(p) =v \right\} \  .
$$
When $\nu \in P_f(\Gamma)$ and $F$ is a finite subset of $P_f(\Gamma)$, set
\begin{equation*}\label{a6}
Z_F(\nu) = Z(\nu) \backslash \left(\bigcup_{\mu \in F} Z(\mu)\right)  \ .
\end{equation*}
 $\mathcal G^0$ is a locally compact Hausdorff space
  in the topology for which the sets of the form $Z_F(\nu)$ is a
  basis of compact and open sets. Then the
elements of $\mathcal G^0 \times \mathbb Z \times \mathcal G^0$ of the form
$$
(\mu x, |\mu| - |\mu'|, \mu'x),
$$
for some $x\in \mathcal G^0, \ \mu,\mu' \in P_f(\Gamma)$, constitute a groupoid $\mathcal G$ with product
$$
(\mu x, |\mu| - |\mu'|, \mu' x)(\nu y, |\nu| -|\nu'|, \nu' y) = (\mu
x, \ |\mu | + |\nu| - |\mu'| - |\nu'|, \nu' y),
$$ 
defined when $\mu' x = \nu y$, and with the involution 
$$
(\mu x, |\mu| - |\mu'|,
\mu'x)^{-1} = (\mu' x, |\mu'| - |\mu|, \mu x) \ .
$$ 
Equipped with the topology for which sets of the form
\begin{equation*}\label{base}
\left\{ \left( \mu x, |\mu| -|\mu'|, \mu'x\right): \ \mu x \in Z_F(\mu), \ \mu'x \in Z_{F'}(\mu') \right\}
\end{equation*}
constitute a basis, $\mathcal G$ is a locally compact, totally disconnected Hausdorff \'etale groupoid. The graph $C^*$-algebra $C^*(\Gamma)$ of $\Gamma$ is then isomorphic to the the associated reduced groupoid $C^*$-algebra, i.e. $C^*(\Gamma) = C^*_r(\mathcal G)$, \cite{Re}. In the following we use the identification $C^*(\Gamma) = C_r^*(\mathcal G)$.

We identify $\mathcal G^0$ with the unit space of $\mathcal G$ via
the embedding
$\mathcal G^0 \ni x \ \mapsto \ ( x, 0,x)$. The range and source maps $r : \mathcal G\to \mathcal G^0$ and $s : \mathcal G \to \mathcal G^0$ of $\mathcal G$ are then the coordinate projections; viz.
$$
r \left(\mu x, |\mu| -|\mu'|, \mu'x\right) = \mu x \ \ \text{and} \ \ s \left(\mu x, |\mu| -|\mu'|, \mu'x\right) = \mu' x \ .
$$

A function $F :  \Gamma_{Ar} \to \mathbb R$ will be called a \emph{potential} on $\Gamma$.  We extend $F$ to a map $F : P_f(\Gamma) \to \mathbb R$ such that $F(v) = 0$ when $v \in \Gamma_V$ and 
$$
F(\mu) = \sum_{i=1}^n F(p_i)
$$
when $\mu = (p_i)_{i=1}^n  \in \left(\Gamma_{Ar}\right)^n$. We can then define a homomorphism (or cocycle) $c_F : \mathcal G \to \mathbb R$ such that
$$
c_F\left(\mu x, |\mu| - |\mu'|, \mu' x\right) = F( \mu) - F(\mu') \  ,
$$
and we obtain a continuous one-parameter group of automorphisms $\alpha^F = \left(\alpha^F_t\right)_{t \in \mathbb R}$ on $C^*(\Gamma)$ defined such that 
\begin{equation*}\label{13-08-form}
\alpha^F_t(f)(\mu x, |\mu| -|\mu'|, \mu' x) = e^{i t\left(F(\mu) - F(\mu')\right)} f(\mu x, |\mu| -|\mu'|, \mu' x)
\end{equation*}
when $f \in C_c(\mathcal G)$, cf. \cite{Re}. An action of this sort is called a \emph{generalized gauge action}; the \emph{gauge action} itself being the one-parameter group corresponding to the constant function $F =1$.

We are here only interested in cases where $C^*(\Gamma)$ is simple, and work by Szymanski gives necessary and sufficient conditions for this to hold. See Theorem 12 in \cite{Sz}. Simplicity of $C^*(\Gamma)$ will be a standing assumption from here on. The first consequence of this is that every KMS weight for a generalized gauge action on $C^*(\Gamma)$ is diagonal in the sense that it factorises through the conditional expectation 
$$
P : C^*(\Gamma) \to C_0(\mathcal G^0),
$$
by Proposition 5.6 in \cite{CT1}. Thus every KMS weight $\psi$ on $C^*(\Gamma)$ is given by a regular Borel measure $m^{\psi}$ on $\mathcal G^0$ in the sense that
\begin{equation}\label{16-06-18}
\psi(a) = \int_{\mathcal G^0} P(a) \ \mathrm{d}m^{\psi}
\end{equation}
for every positive element $a \in C^*(\Gamma)$. As shown in \cite{Th2} the measure $m^{\psi}$ of an extremal $\beta$-KMS weight  $\psi$ is concentrated either on $Q(\Gamma)$ or on $P(\Gamma)$. To obtain a further differentiation let $\Wan(\Gamma)$ denote the set of \emph{wandering} paths $p \in P(\Gamma)$, i.e. the infinite paths $p$ which 'go to infinity' in the sense that for every finite set $F \subseteq \Gamma_V$ there is an $N \in \mathbb N$ such that $s(p_n) \notin F \ \forall n \geq N$. $\Wan(\Gamma)$ is a Borel subset of $P(\Gamma)$ by Lemma 4.13 in \cite{Th4}. When $m^{\psi}$ is concentrated on $P(\Gamma)$ it is either concentrated on the set of wandering paths $\Wan(\Gamma)$ or on its complement. When the latter occurs the measure $m^{\psi}$ is conservative for the action of the shift on $P(\Gamma)$ and its factor type was determined in \cite{Th5}. In this paper we consider the cases where $m^{\psi}$ is concentrated on $\Wan(\Gamma)$.  The measure $m^{\psi}$ is then dissipative with respect to the action of the shift, and we call $\psi$ a \emph{dissipative} KMS weight. The aim is to determine the factor type of $M(\psi)$ when $\psi$ is an extremal dissipative $\beta$-KMS weight and for this we calculate $S\left(M(\psi)\right)$. This will be done by realizing $M(\psi)$ as the von Neumann algebra of a countable Borel equivalence relation and by using results from \cite{FM1} and \cite{FM2}.

\section{The main result}

\subsection{The detour semigroup of a wandering path}

Let $p \in \Wan(\Gamma)$. When $i < j$ we write $p[i,j)$ for the path $p_ip_{i+1}\cdots p_{j-1}$ and $p[i,\infty)$ for the infinite path $p_ip_{i+1}\cdots $. A \emph{detour} on $p$ is a triple $(i,j; \mu)$ where $i < j$ in $\mathbb N$ and $\mu$ is a finite path such that $s(\mu) = s(p_i)$ and $r(\mu) = s(p_j)$. We set
$$
\Delta_p(i,j;\mu) = F\left(p[i,j)\right) - F(\mu) \ .
$$
Given a finite subset $H \subseteq \Gamma_V$ and a finite or infinite path $p$ in $\Gamma$ we write $p \cap H$ for the set of vertexes from $H$ occurring in $p$.  We say that the detour $(i,j;\mu)$ is \emph{outside of} $H$ when $p[i,\infty)\cap H  = \mu \cap H =\emptyset$.

For notational convenience we consider in the following $[0,\infty]$ as a topological space identified in the natural way with the one-point compactification of $[0,\infty)$, and we use the convention that
$$
\frac{1}{0} = \infty \ \ \text{and}  \ \ \frac{1}{\infty} = 0 \ .
$$

Let $\Phi(p)$ be the set of non-negative real numbers $t \in [0,\infty)$ with the property that for all open neighbourhoods $U$ of $t$ and $V$ of $\frac{1}{t}$ in $[0,\infty]$ and for all finite subsets $H \subseteq \Gamma_V$, there are detours $(i,j;\mu)$ and $(i',j';\mu')$ on $p$ outside of $H$ such that 
$$
e^{\Delta_p(i,j;\mu)} \in U 
$$
and
$$
e^{\Delta_p(i',j';\mu')} \in V \ .
$$

\begin{lemma}\label{22-05-18} For $p \in \Wan(\Gamma)$ the set $\Phi(p)$ is a closed sub-semigroup of $\left( [0,\infty), \ \cdot\right) $ containing $1$, and every element of $\Phi(p) \backslash \{0\}$ has an inverse in $\Phi(p)$. 
\end{lemma}
\begin{proof} If $(i,j;\mu)$ and $(i',j';\mu')$ are detours on $p$ and $i' > j$ then 
$$
(i,j'; \mu p[j,i') \mu')
$$ 
is also a detour on $p$, and
$$
e^{\Delta_p\left(i,j'; \mu p[j,i') \mu'\right)} = e^{\Delta_p(i,j;\mu)}  e^{\Delta_p(i',j';\mu')} \ .
$$ 
From this observation it follows in a straightforward way that $\Phi(p)$ is a closed semigroup for which all non-zero elements have an inverse.
That $\Phi(p)$ contains $1$ follows because $(i,j;p[i,j))$ is a detour on $p$ for all $i < j$.
\end{proof}

We call $\Phi(p)$ the \emph{detour semigroup} of $p$. For $s \in \{-1\} \cup [0,1]$ we define the following closed semigroup $S(s)$ in $[0,\infty)$:
\begin{itemize}
\item $S(-1) = \{1\}$,
\item $S(0) = \{0,1\}$,
\item $S(s) = \{0\} \cup \left\{ s^z : \ z \in \mathbb Z\right\}$ for $s \in ]0,1)$, and
\item $S(1)  = [0,\infty)$.
\end{itemize}

\begin{lemma}\label{18-08-18b} Let $S$ be a closed sub-semigroup of $\left( [0,\infty), \ \cdot\right)$ containing $1$ such that every element of $S\backslash \{0\}$ has an inverse in $S$. Then $S = S(s)$ for some $s \in \{-1\} \cup [0,1]$. 
\end{lemma}
\begin{proof} Note that $H = \left\{ \log t : \ t \in S \backslash \{0\} \right\}$ is a closed additive subgroup of $\mathbb R$ and hence either discrete or equal to $\mathbb R$. If $ H = \{0\}$ it follows that $S = \{1\}$ when $0 \notin S$ while $S= \{0,1\}$ when $0 \in S$. If $H = \mathbb Z\alpha$ for some $\alpha > 0$ it follows that $S = \{0\} \cup \left\{ s^z : \ z \in \mathbb Z\right\}$ with $s= e^{-\alpha}$, and when $H = \mathbb R$ it follows that $S = [0,\infty)$. 
\end{proof}

\begin{corollary}\label{18-08-18a} For $p \in \Wan(\Gamma)$ there is a $s \in \{-1\} \cup [0,1]$ such that $\Phi(p) = S(s)$. 
\end{corollary}

We set
$$
\Wan^s(\Gamma) = \left\{ p \in \Wan(\Gamma): \ \Phi(p) = S(s) \right\} \ 
$$
for $s \in \{-1\} \cup [0,1]$.

\begin{lemma}\label{26-08-18} There is a Borel map $\chi : \Wan(\Gamma) \to \{-1\} \cup [0,1]$ such that $\Wan^s(\Gamma) = \chi^{-1}(s)$ for all $s \in \{-1\} \cup [0,1]$. In particular, $\Wan^s(\Gamma)$ is a Borel subset of $P(\Gamma)$ for each $s \in \{-1\} \cup [0,1]$.
\end{lemma}
\begin{proof} We define $\chi : \Wan(\Gamma) \to \{-1\} \cup [0,1]$ such that
$$
\chi(p) = \begin{cases} -1, \ & \ \text{when} \ 0 \notin \Phi(p) \ , \\ \sup [0,1) \cap \Phi(p) , \ & \ \text{when} \ 0 \in \Phi(p) \ . \end{cases}
$$
Since $\Wan^s(\Gamma) = \chi^{-1}(s)$ it suffices to show that $\chi$ is a Borel map. To this end we first prove
\begin{obs}\label{26-08-18a} Let $K$ be a compact subset of $[0,\infty)$. Then
$$
\left\{p \in \Wan(\Gamma): \ \Phi(p) \cap K \neq \emptyset \right\}$$
 is a Borel subset of $P(\Gamma)$.
\end{obs}
To prove this, choose a sequence $H_1 \subseteq H_2 \subseteq H_3 \subseteq \cdots$ of finite subsets of $\Gamma_V$ such that $\bigcup_n H_n = \Gamma_V$. Let $U_1 \supseteq U_2 \supseteq \cdots$ and $V_1 \supseteq V_2 \supseteq \cdots $ be decreasing open subsets in $[0,\infty]$ such that $\overline{U_{k+1}} \subseteq U_k$ and $ \overline{V_{k+1}} \subseteq V_k$ for all $k$, and
$$
\bigcap_k {U_k} = K \ \ \text{and} \ \ \bigcap_k {V_k} = K^{-1} \ ,
$$
where $K^{-1} = \left\{ \frac{1}{x} : \ x \in K \right\}$. For $n,m,k \in \mathbb N$, let $M(n,m,k)$ be the set of elements $p \in \Wan(\Gamma)$ with the property that there are detours $(i,j;\mu)$ and $(i',j'; \mu')$ on $p[m,\infty)$ with $\mu \cap H_n = \mu' \cap H_n = \emptyset$ such that 
\begin{equation*}\label{31-05-18}
e^{\Delta_p(i,j;\mu)} \in U_k \ , 
\end{equation*}
and
\begin{equation*}\label{31-05-18a}
e^{\Delta_p(i',j';\mu')} \in V_k  \ .
\end{equation*}
For each $p \in M(n,m,k)$ there is a finite path $\mu$ such that 
$$
p \in Z(\mu) \cap \Wan(\Gamma) \subseteq M(n,m,k) \ ,
$$ 
and there is therefore an open subset $W$ in $P(\Gamma)$ such that $ W\cap \Wan(\Gamma) = M(n,m,k)$. $\Wan(\Gamma)$ is a Borel subset of $P(\Gamma)$ by Lemma 4.13 in \cite{Th4} and hence so is $M(n,m,k)$. This completes the proof of Observation \ref{26-08-18a} because
$$
\left\{p \in \Wan(\Gamma): \ \Phi(p) \cap K \neq \emptyset \right\} = \bigcap_{n,m,k} M(n,m,k) \ .
$$
We proceed to the proof that $\chi$ is a Borel map: Note that 
$$
\chi^{-1}(-1) = \Wan^{-1}(\Gamma) = \left\{p \in \Wan(\Gamma) : \ 0 \notin \Phi(p) \right\}
$$
is a Borel set by Observation \ref{26-08-18a}. Let $X = \Wan( \Gamma) \backslash \Wan^{-1}(\Gamma)$. It suffices to show that $X \cap \chi^{-1}([0,\alpha])$ is Borel for all $\alpha \in [0,1]$. For this note first that $X \cap \chi^{-1}([0,1]) = X$. We may therefore assume that $0 \leq \alpha < 1$. Then 
$$
X \cap \chi^{-1}([0,\alpha]) = \bigcap_n A_n \ ,
$$
where 
$$
A_n = \left\{ p \in X : \ \left[\alpha + \frac{1}{n}, 1- \frac{1}{n}\right] \cap \Phi(p) = \emptyset \right\}
$$
which is a Borel set by Observation \ref{26-08-18a}. 
\end{proof}

\subsection{Statement of the main result}\label{main}
 Let $\Gamma$ be a directed graph such that $C^*(\Gamma)$ is simple and consider a generalized gauge action $\alpha^F$ on $C^*(\Gamma)$. Let $\psi$ be an extremal dissipative $\beta$-KMS weight for $\alpha^F$ and $M(\psi)$ the corresponding von Neumann algebra factor. Let $m^{\psi}$ be the Borel measure on $\mathcal G^0$ determined by \eqref{16-06-18}. %It follows from Lemma 4.14 in \cite{Th4} that $m^{\psi}$ is concentrated on $\Wan(\Gamma)$.

\begin{proposition}\label{22-04-18e} There is a $s\in \{-1\} \cup [0,1]$ such that $m^{\psi}$ is concentrated on $\Wan^s(\Gamma)$.
\end{proposition}
\begin{proof} Let $\sigma$ denote the shift map on $\Wan(\Gamma)$ and let $\chi : \Wan(\Gamma) \to \{-1\} \cup [0,1]$ be the Borel map from Lemma \ref{26-08-18}. It follows from the definition of the detour semigroup that $\Phi(\sigma(p)) = \Phi(p)$ and hence that $\chi\circ \sigma = \chi$. Define $T : P(\Gamma) \to \{-2,-1\} \cup [0,1]$ such that
$$
T(p) = \begin{cases} -2, & \ p \notin \Wan(\Gamma) \\ \chi(p), & \ p \in \Wan(\Gamma) \ . \end{cases}
$$
$T$ is a Borel map and $T \circ \sigma = T$. Since $m^{\psi}$ is ergodic for $\sigma$ by Theorem 4.15 in \cite{Th4} it follows from Lemma 4.12 in \cite{Th4} that $m^{\psi}$ is concentrated on $T^{-1}(s)$ for some $s\in \{-2,-1\} \cup [0,1]$. Since $m^{\psi}$ is concentrated on $\Wan(\Gamma)$ it follows that $s \in \{-1\} \cup [0,1]$.
\end{proof}

When $\beta =0$ the weight $\tilde{\psi}$ is a densely defined lower-semicontinuous faithful trace on $M(\psi)$ which is therefore semi-finite, i.e. either a type $I$ or type $II$ factor. We consider in the following only the cases $\beta \neq 0$ and we use the convention that $0^{\beta} = 0$. The main result of the paper is

\begin{theorem}\label{22-05-18d} Assume $\beta \neq 0$. Let $s \in \{-1\} \cup [0,1]$ and assume that $m^{\psi}$ is concentrated on $\Wan^s(\Gamma)$. Then
\begin{itemize}
\item $M(\psi)$ is of type $I$ or $II$ when $s= -1$, and
\item $M(\psi)$ is of type $III_{\lambda}$ where $\lambda = s^{|\beta|}$ when $s \in [0,1]$.
\end{itemize}
\end{theorem}

The proof of Theorem \ref{22-05-18d} will be given in the following subsection.

\subsection{Tail equivalence and the proof of the main result}\label{tail}
Fix a vertex $v\in \Gamma_V$.  The projection $P_v = 1_{Z(v)} \in C_c(\mathcal G^0)$ is fixed by $\alpha^F$ and it is full since $C^*(\Gamma)$ is simple and we can therefore consider the state
$$
\phi = \psi(P_v)^{-1} \psi|_{P_vC^*(\Gamma)P_v} ,
$$
which is a $\beta$-KMS state for the restriction of $\alpha^F$ to $P_vC^*(\Gamma)P_v$. This corner $P_vC^*_r(\mathcal G)P_v$ is the $C^*$-algebra of the reduction 
$$
\mathcal G_v = 
\left\{ \left( \mu x, |\mu| -|\mu'|, \mu'x\right) \in \mathcal G: \ s(\mu) = s(\mu') = v \right\} \ ;
$$ 
i.e. $C^*_r\left(\mathcal G_v\right) = P_vC_r^*(\mathcal G)P_v$. Note that $(\Wan(\Gamma),m^{\psi })$ is a standard Borel space. Recall that two infinite paths $p,q \in P(\Gamma)$ are \emph{tail-equivalent} when there are natural numbers $n,m \in \mathbb N$ such that $p_{i+n} = q_{i+m}$ for all $i \geq 0$. We write $p  \sim q$ when this holds. Set
$$
\Wan(Z(v)) = \left\{ p \in \Wan(\Gamma) : \ s(p) = v \right\}
$$
which is a Borel subset of $\Wan(\Gamma)$. On $\Wan(Z(v))$ we consider the Borel probability measure
$$
m = m^{\psi}(Z(v))^{-1}m^{\psi} = \psi(P_v)^{-1} m^{\psi} \ .
$$
Tail-equivalence is then a countable standard equivalence relation on $\left(\Wan(Z(v)), m\right)$ which we identify with the Borel subset
$$
R = \left\{ (p,q) \in \Wan(Z(v)) \times  \Wan (Z(v)) : \ p \sim q \right\} \ 
$$
of $\Wan (Z(v))  \times  \Wan(Z(v))$, cf. \cite{FM1}. Tail-equivalence in $\Wan(Z(v))$ gives therefore rise to a von Neumann algebra ${\bf M}(R)$ as explained in Section 2 of \cite{FM2}. To relate this algebra to $M(\psi)$ observe that when $p,q \in \Wan(\Gamma)$ are tail-equivalent there is a unique integer $k(p,q) \in \mathbb Z$ such that 
$$
p_{i+k(p,q)} = q_i
$$ 
for all sufficiently large $i$. The resulting map $k : R \to \mathbb Z$ is a cocycle, and we can define 
$\Psi : R \to \mathcal G_v$ such that
$$
\Psi(p,q) = (p, k(p,q), q) \ .
$$ 
Then $\Psi$ is a Borel isomorphism and a groupoid isomorphism of $R$ onto the reduction $\mathcal G|_{\Wan(Z(v))}$ of $\mathcal G$. Note that when $f \in C_c(\mathcal G_v)$ the function $f \circ \Psi$ is a left finite function on $R$ as defined in \cite{FM2}, and that $f \to f \circ \Psi$ is a $*$-homomorphism from $C_c(\mathcal G_v)$ into the $*$-algebra of left finite functions of \cite{FM2}. Let $\nu_r$ be the measure on $R$ defined such that
$$
\nu_r(C) = \int_{\Wan(Z(v))} \# \left\{ y \in \Wan(Z(v)): \ (y,x) \in C \right\} \ \mathrm{d} m(x) \ .
$$ 
This is the right counting measure on $R$ defined in \cite{FM1}.
Since the $*$-algebra of left finite functions acts as bounded operators on $L^2(R,\nu_r)$  
we obtain a $*$-homomorphism $\pi : C_c(\mathcal G_v) \to B\left(L^2(R,\nu_r)\right)$. Let $u \in L^2(R,\nu_r)$ be the characteristic function of the diagonal in $R$, i.e.
$$
u(p,q) = \begin{cases} 1 & \ \text{when} \ p = q \\ 0 & \ \text{when} \ p \neq q \ . \end{cases}
$$
 It is straightforward to verify that
\begin{itemize}
\item[i)] $u$ is a cyclic vector for $\pi$, i.e. $\pi\left(C_c(\mathcal G_v)\right)u$ is dense in $L^2(R,\nu_r)$, 
\item[ii)] $\pi\left(C_c(\mathcal G_v)\right)$ is dense in ${\bf M}(R)$ for the strong operator topology, and
\item[iii)] $\left< \pi(f)u,u\right> = \phi(f)$ for all $f \in C_c\left(\mathcal G_v\right)$.
\end{itemize}
Let $\left(\pi_{\phi}, H_{\phi},u_{\phi}\right)$ be the GNS-triple arising from the state $\phi$. It follows from i) and iii) that there is a unitary $W : H_{\phi} \to L^2(R,\nu_r)$ such that $W^*\pi_{\phi}(f)W = \pi(f)$ for all $f \in C_c(\mathcal G_v)$. Combined with ii) this gives us an isomorphism
\begin{equation*}\label{14-08-18}
\pi_{\phi}\left(P_vC^*(\Gamma)P_v\right)'' \simeq {\bf M}(R) \ 
\end{equation*}
when we use the identification $C^*(\mathcal G_v) = P_vC^*(\Gamma)P_v$. Since 
$$
\pi_{\phi}\left(P_vC^*(\Gamma)P_v\right)'' \simeq \pi_{\psi}(P_v)M(\psi)\pi_{\psi}(P_v)
$$ 
by Lemma 2.3 in \cite{Th5} it follows that $ {\bf M}(R)$ is a corner in $M(\psi)$; more precisely,
$$
\pi_{\psi}(P_v)M(\psi)\pi_{\psi}(P_v) \simeq  {\bf M}(R) \ .
$$
It follows then from \cite{Co}, Corollaire 3.2.8(b), that 
\begin{equation}\label{15-08-18}
S\left(M(\psi)\right) = S\left(\pi_{\psi}(P_v)M(\psi)\pi_{\psi}(P_v)\right) = S\left( {\bf M}(R)\right) \ .
\end{equation}
We define a cocycle $c : R \to \mathbb R$ such that
$$
c(p,q) = \lim_{n \to \infty} F(p[1,n+k(p,q)]) - F(q[1,n]) \ .
$$
Observe that $c_F \circ \Psi = c$. The Radon-Nikodym derivative $D_{\beta}$ of $m$ with respect to $R$, as defined in Definition 2.1 of \cite{FM1}, is given by
$$
D_{\beta}(p,q) = e^{-\beta c(p,q)} \ .
$$ 
By combining \eqref{15-08-18} with Proposition 2.11 in \cite{FM2} we conclude that $S\left(M(\psi)\right) \backslash \{0\}$ is the essential range $r^*(D_{\beta})$ of $D_{\beta}$. It follows that
\begin{equation*}\label{20-08-18}
S\left(M(\psi)\right)  \backslash \{0\} = r^*(D_{\beta}) = r^*(D_1)^{\beta} = \left\{ x^{\beta} : \ x \in r^*(D_1)\right\} \ .
\end{equation*}
Combined with Lemme 3.1.2 in \cite{Co} we see that
\begin{equation}\label{22-08-18e}
S\left(M(\psi)\right)   = \left\{ x^{\beta} : \ x \in r^*(D_1)\right\} \cup \{0\} \ 
\end{equation}
unless $M(\psi)$ is semi-finite in which case $S(M(\psi)) = \{1\}$.

\bigskip

\emph{Proof of Theorem \ref{22-05-18d}:} The proof will be divided into smaller steps, organized as lemmas. Since $m^{\psi}$ is concentrated on $\Wan^s(\Gamma)$ by assumption we have that $\Phi(p) = S(s)$ for $m^{\psi}$-almost all $p \in \Wan(\Gamma)$.

\begin{lemma}\label{24-08-18} Let $t \in [0,\infty) \backslash S(s)$ and let $U$ be an open neigbourhood of $t$ in $[0,\infty)$. There is a sequence $B_n$ of Borel sets in $\Wan(Z(v))$ such that 
\begin{itemize}
\item $m(B_n) > 0$ for all $n$,
\item $m\left(\bigcup_n B_n\right) = 1$,
\item for all $n$ there is an open neighboúrhood $U_n$ of $t$ such that $e^{c(x,y)} \notin U_n$ for all $(x,y) \in R \cap \left(B_n \times B_n\right)$.
\end{itemize}
\end{lemma}
\begin{proof} Let $H_1 \subseteq H_2  \subseteq H_3 \subseteq \cdots$ be a sequence of finite sets in $\Gamma_V$ such that $\Gamma_V = \bigcup_n H_n$. Let $U_1 \supseteq U_2 \supseteq \cdots$ be a decreasing neighbourhood base at $t$ in $[0,\infty)$ and set
$$
U_k^{-1} = \left\{ \frac{1}{x} : \ x \in U_k \right\} \ .
$$
For $a,b,c \in \mathbb N$ let $A(a,b,c,+)$ (resp. $A(a,b,c,-)$) be the set of elements $p \in \Wan(Z(v))$ such that $p[a,\infty) \cap H_b = \emptyset$ and for all detours $(i,j;\mu)$ on $p[a,\infty)$ outside of $H_b$ we have that $e^{\Delta_p(i,j;\mu)} \notin U_c$ (resp. $e^{\Delta_p(i,j;\mu)} \notin U_c^{-1}$). $A(a,b,c,\pm)$ is a relatively closed subset of $\Wan(Z(v))$ and hence a Borel set. The collection of sets 
$$
A(\nu, a,b,c, \pm) = Z(\nu) \cap \sigma^{-|\nu|}\left( \sigma^a \left(A(a,b,c, \pm)\right)\right) \ , 
$$
where $\nu$ runs over all $\nu \in P_f(\Gamma)$ with $s(\nu) = v$ and $a,b,c \in \mathbb N$, is a countable collection of Borel sets in $\Wan(Z(v))$.
Since $t \notin S(s)$,
$$
\Wan^s(\Gamma) \cap \Wan(Z(v))  \subseteq \bigcup_{a,b,c,\pm} A(a,b,c,\pm) \ .
$$
Since
$$
A(a,b,c,\pm) \ \ = \  \bigcup_{\nu \in P_f(\Gamma) , \ s(\nu) = v} \ A(\nu, a,b,c, \pm) \ ,
$$
and since $m$ is concentrated on $\Wan^s(\Gamma) \cap \Wan(Z(v))$ it suffices to check that 
\begin{equation*}\label{24-08-18a}
e^{c(x,y)} \notin U_c 
\end{equation*}
when $(x,y) \in R \cap \left( A(\nu, a,b,c, \pm) \times A(\nu, a,b,c,\pm)\right)$. Let therefore $(x,y)$ be an element in $R \cap \left( A(\nu, a,b,c, -) \times A(\nu, a,b,c,-)\right)$. There are $n,k > |\nu|+1$ such that $x_{i + n} = y_{i+k}$ for all $i \geq 0$. Then
$$
(|\nu|+1,k;x[|\nu|+1,n))
$$
is a detour on $y[a,\infty)$ outside of $H_b$ and hence
$$
e^{\Delta_y(|\nu|+1,k;x[|\nu|+1,n))} \notin U_c^{-1} \ .
$$
Since $c(x,y) = - \Delta_y(|\nu|+1,k;x[|\nu|+1,n))$ it follows that $e^{c(x,y)} \notin U_c$. A similar argument works when $(x,y) \in R \cap \left( A(\nu, a,b,c, +) \times A(\nu, a,b,c,+)\right)$; just use the detour $(|\nu|+1,n;y[|\nu|+1,k))$ on $x[a,\infty)$ instead.

\end{proof}

\begin{lemma}\label{18-08-18d} $r^*(D_1) \subseteq S(s) \subseteq r^*(D_1) \cup \{0\}$ \ .
\end{lemma}

\begin{proof} The first inclusion: Let $t\in [0,\infty) \backslash S(s)$ and let $B_n$ be any of the Borel subsets from the sequence in Lemma \ref{24-08-18}. It follows from Proposition 8.4 in \cite{FM1} that $t \notin r^*(D_1)$.

The second inclusion: Let $t \in S(s)$ and let $U$ be a neighbourhood of $t$ in $[0,\infty)$. When $A \subseteq \Wan(Z(v))$, set
$$
R(A) = \left\{ (x,y)  \in R : \ x, y \in A, \  e^{c(x,y)} \in U \  \right\} \ 
$$
and
$$
A' = A \cap \Wan^s(\Gamma) \ .
$$
We denote the two coordinate projections $R \to \Wan(Z(v))$ by $\pi_l$ and $\pi_r$; i.e. $\pi_l(x,y) =x$ and $\pi_r(x,y) = y$. 
The argument for the second inclusion is based on the following observation which we shall also need later.

\begin{obs}\label{21-08-18} Let $B \subseteq \Wan(Z(v))$ be a Borel subset. Then
$$
m\left(\pi_l(R(B)\right) = m\left(\pi_r(R(B)\right) = m(B) \ .
$$
\end{obs}
\emph{Proof of Observation \ref{21-08-18}:} Let $p \in \Wan^s(\Gamma) \cap \Wan(Z(v))$. Then $\Phi(p) = S(s)$ and it follows that for every $N \in \mathbb N$ there is a $p' \in \Wan^s(\Gamma)$ and $k,n > N$ such that $p'_i = p_i, \ i  \leq N$, $p_{n+j} = p'_{k+j}, \ j \geq 0$, and
$$
e^{F(p[N,n)) - F(p'[N,k))} \in U \ .
$$
This shows that
\begin{equation}\label{13-08-18bb}
\pi_l\left(R(V')\right) = V'
\end{equation}
for every relatively open set $V \subseteq \Wan(Z(v))$. By applying the same argument with $U$ replaced by an open neighborhood of $\frac{1}{t}$ in $[0,\infty]$ it follows in the same way that
\begin{equation}\label{13-08-18bbb}
\pi_r\left(R(V')\right) = V'
\end{equation}
for every relatively open set $V \subseteq \Wan(Z(v))$. Now choose  a decreasing sequence $V_1 \supseteq V_2 \supseteq \cdots$ of relatively open sets in $\Wan(Z(v))$  such that $B \subseteq \bigcap_k V_k$ and $\lim_{k \to \infty} m(V_k) = \lim_{k \to \infty} m(V_k') = m(B)$. Then
$$
\bigcap_k R(V'_k) = R(B') \cap X \ ,
$$
where $X$ is a Borel set such that 
$$
X  \subseteq {\pi_l}^{-1} \left( \bigcap_k V'_k \backslash B'\right) \cup  {\pi_r}^{-1} \left( \bigcap_k V'_k \backslash B'\right) \ . 
$$
Consider the Borel measures $\nu_r$ and $\nu_l$ on $R$ defined from $m$ as in Theorem 2 in \cite{FM1}; $\nu_r$ is the right counting measure used above. Since $m( \bigcap_k V'_k \backslash B') = 0$ it follows from the definitions that
$$
\nu_r \left(  {\pi_r}^{-1} \left( \bigcap_k V'_k \backslash B'\right)\right) = \nu_l \left(  {\pi_l}^{-1} \left( \bigcap_k V'_k\backslash B'\right)\right) = 0 \ .
$$
Since $\nu_l$ and $\nu_r$ are equivalent by Theorem 2 (c) in \cite{FM1} (a fact we have already used to obtain $D_{\beta}$) it follows that
$\nu_l \left(  {\pi_r}^{-1} \left( \bigcap_k V'_k \backslash B'\right)\right) = 0$, and hence that $\nu_l(X) = \nu_r(X) = 0$. Thus 
$$
\nu_i(R(B')) = \nu_i(\bigcap_k R(V'_k)) = \lim_{k \to \infty} \nu_i(R(V'_k)), \ \ i \in \{r,l\} \ . 
$$
Since
\begin{equation*}
\begin{split}
&m\left( \pi_l\left(R(V'_k)\right) \backslash \pi_l\left(R(B')\right)\right) \leq m\left(\pi_l\left(R(V'_k)\backslash R(B') \right)\right) \\
& \\
& \leq \int_{ \pi_l\left(R(V'_k)\backslash R(B')\right)} \# \left\{ y : \ (x,y) \in R(V'_k) \backslash R(B')\right\} \ \mathrm{d} m(x) \\
& \\
& \leq \nu_l \left(R(V'_k)\backslash R(B')\right)  \ ,
\end{split}
\end{equation*} 
we find that $m(\pi_l(R(B')) = \lim_{k \to \infty} m\left(\pi_l(R(V'_k))\right)$. Since $m(B) = \lim_{k \to \infty} m(V'_k)$ it follows from \eqref{13-08-18bb} that $m\left(\pi_l(R(B))\right) = B$. By using \eqref{13-08-18bbb} we find in the same way that $m\left(\pi_r(R(B))\right) = B$. \qed

Combining Observation \ref{21-08-18} with Proposition 8.4 in \cite{FM1} it follows that $S(s) \backslash \{0\} \subseteq r^*(D_1)$. 
\end{proof}

\begin{lemma}\label{22-08-18} $M(\psi)$ is semi-finite iff $0 \notin S(s)$.
\end{lemma}
\begin{proof} It follows from Th\'eor\`eme 1.3.4 in \cite{Co} that $M(\psi)$ is semi-finite iff ${\bf M}(R)$ is semi-finite iff the modular automorphism group on ${\bf M}(R)$ coming from the vector state $u$ is inner. So when $M(\psi)$ is semi-finite it follows from Proposition 2.8 in \cite{FM2} that the one-parameter group $\theta$ on ${\bf M}(R)$ with the property that
$$
\theta_t(a)(x,y) = e^{-i\beta t c(x,y)}a(x,y)
$$
for all left finite elements $a$ is inner. Since $L^{\infty}(\Wan(Z(v)))$ is maximal abelian in ${\bf M}(R)$ and contained in the fixed point algebra of $\theta$ it follows that there is a Borel function $H : \Wan(Z(v)) \to \mathbb R$ such that 
\begin{equation}\label{22-08-18a}
c(x,y) = H(x) - H(y)
\end{equation}
for $\nu_r$-almost all $(x,y) \in R$. (Recall that $\beta \neq 0$.)  For some $R > 0$ the set
$ B =H^{-1}\left([-R, R]\right)$ has positive $m$-measure. Assume for a contradiction that $0 \in S(s)$. With $U = [0,e^{-3R})$ it follows from Observation \ref{21-08-18} that $\nu_r(R(B)) > 0$ and there are therefore elements $x,y \in B$ for which \eqref{22-08-18a} holds and at the same time $e^{c(x,y)} \leq e^{-3R}$. This contradiction shows that $0 \notin S(s)$. For the converse assume that $0 \notin S(s)$. It follows from Lemma \ref{24-08-18} that there is a Borel subset $B \subseteq \Wan(Z(v))$ such that $m(B) > 0$ and an $\epsilon \in (0,1)$ such that $e^{c(x,y)}\notin [0,\epsilon)$ when $(x,y) \in R \cap (B\times B)$. Then
$$
\epsilon \leq e^{c(x,y)} \leq \epsilon^{-1}
$$
for all $(x,y) \in R \cap (B\times B)$, showing that $c$ is bounded on $R \cap (B \times B)$. It follows from Theorem 6 in \cite{FM2} that there is a Borel function $H : B \to \mathbb R$ such that \eqref{22-08-18a} holds for $\nu_r$-almost all $(x,y) \in R \cap \left(B \times B\right)$. Therefore the restriction of $\theta$ to the non-zero corner $1_B{\bf M}(R)1_B$ is inner, and Th\'eor\`eme 1.3.4 in \cite{Co} implies that $1_B{\bf M}(R)1_B$ is semi-finite. Therefore ${\bf M}(R)$ and $M(\psi)$ are also semi-finite.
\end{proof}

\begin{lemma}\label{22-08-18c} $S(M(\psi)) = S(s)^{\beta}$.
\end{lemma}
\begin{proof} When $s = -1$, $S(s) = \{1\}$ and $M(\psi)$ is semi-finite by Lemma \ref{22-08-18} and hence $S(M(\psi)) = \{1\}$ by Lemme 3.1.2 in \cite{Co}. It follows that $S(s)^{\beta} = \{1\} = S(M(\psi))$ in this case. When $s \in [0,1]$, $0 \in S(s)$ and hence $M(\psi)$ is not semi-finite by Lemma \ref{18-08-18d}. Furthermore, $S(s) = r^*(D_1) \cup \{0\}$ by Lemma \ref{18-08-18d}.   Therefore the conclusion follows from \eqref{22-08-18e}.
\end{proof}

\emph{Proof of Theorem \ref{22-05-18d}:} The case $s = -1$ was covered in Lemma \ref{22-08-18}, and the same lemma implies that $M(\psi)$ is of type $III$ when $s \in [0,1]$. To complete the proof of Theorem \ref{22-05-18d} it suffices therefore now to compare Lemma \ref{22-08-18c} with the definition of $III_{\lambda}$ in Section IV of \cite{Co}. We leave that to the reader. \qed

\section{Applications and examples}
\subsection{Cayley graphs}
 Given a group $G$ and a finite set of generators $Y$ of $G$ there is a natural way to define a directed graph $\Gamma_{G,Y}$ whose vertexes are the elements of $G$ and with an arrow from $g \in G$ to $h \in G$ iff $g^{-1}h \in Y$. This is the \emph{Cayley graph}. It was pointed out in \cite{CT2} that the graph $C^*$-algebra of $\Gamma_{G,Y}$ is the universal $C^*$-algebra
generated by a set 
$v(g,s), \ g \in G, \ s \in Y$, of partial isometries such that
\begin{equation*}\label{CKrel2}
v(h,t)^*v(g,s) = \begin{cases} 0 \ & \ \text{when} \ h \neq g \ \text{or} \ t \neq s, \\
\sum_{y \in Y} v(gy,y)v(gy,y)^* \ & \text{when} \ h = g \ \text{and} \ t = s  \ .
\end{cases}
\end{equation*}
Following \cite{CT2} we denote the corner $P_{e_0}C^*\left(\Gamma_{G,Y}\right)P_{e_0}$ defined by the neutral element $e_0 \in G$ by $O_Y(G)$, partly because it is a $C^*$-subalgebra of the Cuntz-algebra $O_n$ where $n = \# Y$, cf. \cite{CT2}. As in \cite{CT2} we assume that $Y$ contains at least two elements to ensure that $C^*\left(\Gamma_{G,Y}\right)$ and $O_Y(G)$ are simple $C^*$-algebras. As in \cite{CT2} we consider the generalized gauge action $\alpha^F$ on $C^*\left(\Gamma_{G,Y}\right)$ given by a function $F : Y \to \mathbb R$. In terms of the generators above it is given by the formula
$$
\alpha^F_t(v(g,s)) = e^{i tF(s)} v(g,s).
$$
We denote the restriction of $\alpha^F$ to $O_Y(G)$ by $\gamma^F$. To describe the factor type of the extremal KMS states for $\gamma^F$ we need the following notation.
When $y = (y_1,y_2, \cdots ,y_n) \in Y^n$ we denote the element $ y_1y_2\cdots y_n \in G$ by $\overline{y}$ and set
$$
F(y) = \sum_{i=1}^n F(y_i) \ .
$$
It follows from Lemma 10.4 and Lemma 10.7 in \cite{Th4} that since $\Gamma_{G,Y}$ is strongly connected, the existence of a $\beta$-KMS weight for $\alpha^F$ requires that $\beta \neq 0$ and that either $F(y) > 0$ for all $y \in \bigcup_n Y^n$ with $\overline{y} = e_0$ or that $F(y) < 0$ for all $y \in \bigcup_n Y^n$ with $\overline{y} = e_0$. Therefore, by Theorem 2.4 in \cite{Th2}, there can also only be a $\beta$-KMS state for $\gamma^F$ when this holds. In particular, it follows that
$$
\theta_F  = \sup \  \left(\left]-\infty, 0\right[ \cap  \left\{ F(y) - F(z) : \ y,z \in \bigcup_n Y^n, \ \overline{y} = \overline{z} = e_0 \right\} \right) 
$$
is defined when there is a $\beta$-KMS state for $\gamma^F$.

\begin{theorem}\label{nov9} Let $\phi$ be an extremal $\beta$-KMS state
  for $\gamma^F$ on $O_Y(G)$. Then $\beta \neq 0$ and $\pi_{\phi}\left(O_Y(G)\right)''$ is of type $III_{\lambda}$ where $\lambda = e^{\theta_F |\beta|}$
  \end{theorem}
\begin{proof} By Theorem 2.4 in \cite{Th2} $\phi$ is the restriction to $O_Y(G)$ of a $\beta$-KMS weight for $\alpha^F$. When that weight is dissipative Theorem \ref{22-05-18d} yields the stated conclusion by observing that for any wandering path $p$ in $\Gamma_{G,Y}$ the detour semigroup $\Phi(p)$ is the closure in $[0,\infty)$ of
$$
\left\{ e^{F(y) - F(z)}: \  y,z \in \bigcup_n Y^n, \ \overline{y} = \overline{z} = e_0 \right\}  \ .
$$
When the weight is conservative the conclusion follows from Theorem 3.3 in \cite{Th5}.
\end{proof}

We note that in the dichotomy between conservative and dissipative KMS weights which was used in the proof above, the presence of a conservative KMS weight is rare. Indeed, when $G$ is infinite a conservative KMS weight can only exist when $G$ has a subgroup of finite index which is isomorphic to $\mathbb Z$ or $\mathbb Z^2$. This follows from the connection to random walks which was explained in Section 3.1 of \cite{Th4} by appeal to Theorem 3.24 on page 36 in \cite{Wo}. Furthermore, a conservative $\beta$-KMS weight can only exist for one value of $\beta$ by Proposition 10.8 in \cite{Th4}.

\subsection{Exits and Araki-Woods factors}\label{exits}

As shown in \cite{Th5} the factor type of a conservative KMS weight is always $III_{\lambda}$ for some $\lambda \in ]0,1]$, and it follows from Theorem \ref{nov9} that the same holds for the actions on Cayley graph considered in \cite{CT2}. In contrast we show in this section that the dissipative KMS weights for generalized gauge actions on graphs can also be of type $I$, type $II$ and type $III_0$.

For convenience we assume for this that $\Gamma$ is strongly connected and row-finite; besides the standing assumption that $C^*(\Gamma)$ is simple. Let $F : \Gamma_{Ar} \to \mathbb R$ be a potential and $\beta \in \mathbb R$ a real number. We define a matrix $A(\beta) = \left(A(\beta)_{v,w}\right)_{v,w \in \Gamma_V}$ over the vertexes of $\Gamma$ such that
$$
A(\beta)_{v,w} = \sum_a e^{-\beta F(a)}
$$
where the sum is over $a \in s^{-1}(v) \cap r^{-1}(w)$, i.e. the arrows that go from $v$ to $w$. It follows that the $\beta$-KMS weights for $\alpha^F$ are dissipative if and only if $A(\beta)$ is transient in the sense that
\begin{equation}\label{27-06-18}
\sum_{n=0}^{\infty} A(\beta)^n_{v,v} \ < \ \infty 
\end{equation}
for one (and hence all) $v \in \Gamma_V$, cf. Theorem 4.10 in \cite{Th2}. We assume therefore here that this holds. As in \cite{Th2} an \emph{exit path} in $\Gamma$ is a sequence
$t=(t_i)_{i=1}^{\infty}$ of vertexes such that there is an arrow $a_i$ with
$s(a_i) = t_i$ and $r(a_i) = t_{i+1}$ for all $i$, and such that $\lim_{i
  \to  \infty} t_i = \infty$ in the sense that $t_i$ eventually leaves
any finite subset of vertexes. An \emph{exit} is a tail-equivalence class of exit paths. For a given exit path
$t$ and a real number $\beta$ we set
$$
t^{\beta}(i) \ = \ A(\beta)_{t_1,t_2} A(\beta)_{t_2,t_3}\cdots A(\beta)_{t_{i-1},t_i} \ .
$$
As in \cite{Th2} we say that $t$ is \emph{$\beta$-summable} when the limit
$$
\lim_{i \to \infty} t^{\beta}(i)^{-1} \sum_{n=0}^{\infty} A(\beta)^n_{v,t_i} 
$$
is finite for one, and hence for all vertexes $v \in \Gamma_V$. As shown in
\cite{Th2} a $\beta$-summable exit gives rise to an $e^{\beta F}$-conformal measure $m_t$ on $P(\Gamma)$ determined by the condition that
$$
m_t(Z(v)) = \lim_{i \to \infty} t^{\beta}(i)^{-1} \sum_{n=0}^{\infty}
A(\beta)^n_{v,t_i} 
$$
for every vertex $v$. It follows from Corollary 5.4 in \cite{Th2} that
the corresponding KMS weight $\psi_t$ on $C^*(\Gamma)$ is extremal. To determine the factor type of $\psi_t$, set 
$$
k_n = \# s^{-1}(t_n) \cap
r^{-1}(t_{n+1}) \ .
$$
Choose a numbering $e^n_1,e^n_2, \cdots , e^n_{k_n}$ of the arrows in
$s^{-1}(t_n) \cap r^{-1}(t_{n+1})$, and let $\omega_n$ be the state on $M_{k_n}(\mathbb C)$
given by
\begin{equation*}\label{nov26}
\omega_n (a) = \frac{\Tr \left( e^{-\beta H}
    a\right)}{\Tr\left(e^{-\beta H}\right)},
\end{equation*}
where $H = \diag\left(F(e^n_1),F(e^n_2), \cdots,
  F(e^n_{k_n})\right)$. Then $\omega = \otimes_{n=1}^{\infty}
\omega_n$ is a state on the UHF-algebra
\begin{equation*}\label{UHF}
A = \otimes_{n=1}^{\infty} M_{k_n}(\mathbb C) ,
\end{equation*}
and we set
$$
\mathcal R(t) = \pi_{\omega}(A)'' .
$$ 
Note that $\mathcal R(t)$ is an Araki-Woods factor, cf. \cite{AW}. As shown in Section \ref{tail} the corner $\pi_{\psi_t}\left(P_{t_1}\right)M(\psi_t)\pi_{\psi_t}\left(P_{t_1}\right)$ can be identified with the Feldman-Moore factor ${\bf M}(R)$ where $R$ is the tail-equivalence relation on $(\Wan(Z(t_1)),m)$ with $m = \psi_t\left(P_{t_1}\right)^{-1}m^{\psi_t}$. Set 
$$
  B = \left\{ p \in P(\Gamma): \ s(p_i) = t_i, \ i =1,2,3, \cdots \right\} \ 
  $$
which is a compact subset of $\Wan(Z(t_1))$ homeomorphic to the infinite product space
$$
X = \prod_{n=1}^{\infty} s^{-1}(t_n) \cap r^{-1}(t_{n+1}) \ .
$$
Note that $m^{\psi_t}(B) > 0$ by Lemma 5.4 in \cite{Th2}.
Under the identification $B =X$ the Borel probability measure on $X$ obtained by normalizing the restriction of $m^{\psi_t}$ to $B$ will be the infinite product Bernoulli measure $\prod_{n=1}^{\infty} \omega_n$, where $\omega_n$ is given by the probability vector obtained by normalizing the vector
$\left(e^{-\beta F(e^n_1)},e^{-\beta F(e^n_2)}, \cdots,
  e^{-\beta F(e^n_{k_n})}\right)$. Hence 
 \begin{equation*}\label{15-08-18h}
 1_B{\bf M}(R)1_B \simeq \mathcal R(t) \ .
 \end{equation*}
In particular, we have shown

\begin{lemma}\label{nov24} Assume that $t$ is $\beta$-summable and let $M(\psi_t)$ be the von Neumann algebra factor of $\psi_t$.
 There is a projection $p \in  M(\psi_t)$ fixed by the modular automorphism group of $\tilde{\psi_t}$ such
  that 
$$
p M(\psi_t)p \simeq \mathcal R(t) \  .
$$ 
\end{lemma}

We will say that the exit path $t = (t_n)_{n=1}^{\infty}$ is \emph{slim}
when 
$$
\lim_{i\to \infty} \left(\# s^{-1}(t_i) \cap r^{-1}(t_{i+1})\right) \ = \ 1 \ .
$$

\begin{proposition}\label{nov25}  Assume that $t$ is $\beta$-summable for the gauge action and let $\psi_t$ be the
  corresponding $\beta$-KMS weight for the gauge action on
  $C^*(\Gamma)$. The factor type of $\psi_t$ is $I_{\infty}$
  when $t$ is slim and $II_{\infty}$ factor when it is
  not.
\end{proposition}
\begin{proof} The gauge action arises by choosing $F$ to be constant
  $1$. It follows that $\mathcal R(t)$ is a type I factor when $t$ is
  slim and the hyper finite $II_1$ factor otherwise. It follows
  therefore from Lemma \ref{nov24} that $
  M(\psi_t)$ is type $I$ when $t$ is slim and type $II$
  otherwise. To complete the proof we need to show that $M(\psi_t)$ is not finite. To this
end choose for each $n \geq 2$ a path $\mu_n$ in $\Gamma$ such that
$s(\mu_n) = t_n$ and $r(\mu_n) = t_1$. The characteristic function
$v_n$ of
the set
$$
\left\{ (\mu_n x, |\mu_n|, x) : \ x \in Z(t_1) \right\}
$$
is an element of $C_c(\mathcal G)$ such that $v_n^*v_n = 1_{Z(t_1)}$ and
$v_nv_n^* \leq 1_{Z(t_n)}$. In particular, the $v_n$'s represent non-zero
partial isometries in $M(\psi_t)$. Since $t$ is an exit an infinite
sub-collection of the $v_n$'s will have orthogonal ranges. It follows that $M(\psi_t)$
is not finite.
\end{proof}

It follows from Proposition \ref{nov25} that the examples considered in \cite{Th2} exhibit extremal dissipative KMS weights of both type $I_{\infty}$ and type $II_{\infty}$. In fact, since all extremal dissipative $\beta$-KMS weights of a strongly connected graph with at most countably many exits arise from $\beta$-summable exits, it follows that they are all of semi-finite type when the action is the gauge action. This illustrates the fundamental difference between Cayley graphs and the graphs with countably many exits.

\begin{example}\label{dec5b} In the final example we show that exits can also give rise to dissipative KMS weights of type $III_0$. For this we consider the following row-finite digraph $\Gamma$.
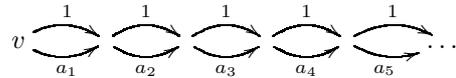
\begin{equation*}\label{G5}
\begin{xymatrix}{
v \ar@/^/[r]^-1\ar@/_/[r]_-{a_1} &  \ar@/^/[r]^-1\ar@/_/[r]_-{a_2}  &  
\ar@/^/[r]^-1\ar@/_/[r]_-{a_3} &  \ar@/^/[r]^-1\ar@/_/[r]_-{a_4} & 
\ar@/^/[r]^-1\ar@/_/[r]_-{a_5} & \hdots }
 \end{xymatrix}
\end{equation*}
The label on an arrow shows the value which the potential $F$ takes on the arrow. To identify $F$ we need only to specify the sequence $\{a_i\}$. Pick first an arbitrary sequence $\lambda_k, k =1,2,3, \cdots$, in the open interval $(0,1)$ and set $a_k = 1 - \log \lambda_k$. It follows from Theorem 2.7 in \cite{Th2} that for any $\beta \in \mathbb R$ there is a $\beta$-KMS weight on $C^*(\Gamma)$ which is unique up to scalar multiplication. Fix $\beta \in \mathbb R$. When we use the obvious identification $\Wan(Z(v)) \simeq \{1,2\}^{\mathbb N}$ the Borel probability measure $m$ from Section \ref{tail} is the product measure $\prod_i m_i$ where $m_i$ is the probability measure on $\{1,2\}$ defined from the probability vector 
$$
\left( \frac{e^{-\beta}}{ e^{-\beta} + e^{-a_k \beta}}, \ \frac{e^{-a_k \beta}}{ e^{-\beta} + e^{-a_k \beta}} \right) \  .
$$ 
For $\beta =1$ this is the vector
$$
\left( \frac{1}{1+\lambda_k}, \  \frac{\lambda_k}{1+\lambda_k}\right) \ .
$$
It follows therefore from Theorem 1.4 in \cite{GS} that we may choose the sequence $\{\lambda_k\}$ such that factor $\mathcal R(t)$ is of type $III_0$; in fact in such a way that the $T$-invariant $T(\mathcal R(t))$ of Connes, \cite{Co}, is any given countable subgroup of $\mathbb R$. 

The graph $\Gamma$ above is not strongly connected, but this can be arranged by adding return paths to $v$ in the same way as in Sections 7 and 8 of \cite{Th2}. With a little care the $C^*$-algebra of the resulting graph will still carry a generalized gauge action with an essentially unique $1$-KMS weight whose factor type is $III_0$ and whose $T$-invariant is any given countable subgroup of $\mathbb R$.

\end{example}

%Text with citations \cite{RefB} and \cite{RefJ}.
%\subsection{Subsection title}
%\label{sec:2}
%as required. Don't forget to give each section
%and subsection a unique label (see Sect.~\ref{sec:1}).
%\paragraph{Paragraph headings} Use paragraph headings as needed.
%\begin{equation}
%\end{equation}

%\begin{figure}
%  \includegraphics{example.eps}
%\end{figure}
%\begin{figure*}
%  \includegraphics[width=0.75\textwidth]{example.eps}
%\end{figure*}
%\begin{table}
%\begin{tabular}{lll}
%\hline\noalign{\smallskip}
%\end{tabular}
%\end{table}

\begin{acknowledgements} I am grateful to Nadia Larsen and Johannes Christensen for helpful discussions. The work was supported by the DFF-Research Project 2 `Automorphisms and Invariants of Operator Algebras', no. 7014-00145B.

\end{acknowledgements}



\end{document}